\documentclass[12pt]{article} 
\usepackage{graphicx}
\usepackage{amsfonts,amsmath, amssymb, amstext, amsthm}
\DeclareGraphicsExtensions{ps,eps,pdf}
\newcommand{\R}{{\mathbb R}}
\newcommand{\C}{{\mathbb C}}

\newcommand{\Z}{{\mathbb Z}}
\renewcommand{\d}{{\partial}}

\renewcommand{\H}{{\mathcal H}}
\newcommand{\T}{{\mathcal T}}

\theoremstyle{plain}

\newtheorem{proposition}{Proposition}

\newtheorem{corollary}{Corollary}

\theoremstyle{remark}
\newtheorem{remark}{Remark}
\newtheorem{example}{Example}
\newtheorem{question}{Question}

\theoremstyle{definition}
\newtheorem{definition}{Definition}

\title{Local scales on curves and surfaces}
\author{Triet Le \thanks{Department of Mathematics, University of Pennsylvania, David Rittenhouse Lab.
209 South 33rd Street, Philadelphia, PA 19104, trietle@math.upenn.edu}}
\begin{document}
\maketitle
\thispagestyle{plain}

\abstract{In this paper, we extend the study of local scales of a function \cite{jones2009local} to studying local scales on curves and surfaces. In the case of a function $f$, the local scales of $f$ at $x$ is computed by measuring the deviation of $f$ from a linear function near $x$ at different scales $t$'s. In the case of a $d$-dimensional surface $\Gamma$, the analogy is to measure the deviation of $\Gamma$ from a $d$-plane near $x$ on $\Gamma$ at various scale $t$'s. We then apply the theory of singular integral operators on $\Gamma$ to show useful properties of local scales. We will also show that the defined local scales are consistent in the sense that the number of local scales are invariant under dilation.}

\section{Introduction}
Given a bounded function $f$ (an image) defined on $\R^n$, an
important task in image analysis is the extraction of local features
and information, which we call local scales. The knowledge of local
scales is then used for tasks like image matching, texture
segmentation, object recognitions, and image and domain decompositions. A common approach in studying local scales (\cite{lindeberg1993detecting}, \cite{lindeberg1998fda}, \cite{strong2003edge}, \cite{lowe2004distinctive}, \cite{brox2004tfb}) in
images is to have a multiscale representation $\{u(t)\}_{t\ge 0}$
(linear or nonlinear) of $f$. The local scales of $f$ at $x$ is
computed using the information coming from $\{u(x,t)\}$, for $t\ge 0$. In \cite{luo2006characteristic}, $\{u(x,t)\}$ is replaced by a the set of shapes $\{S(x,t)\}$ that contains $x$. These methods compute a single meaningful scale at each location $x$. We refer the readers to \cite{le2010local} for an over view and analysis of local scales in images. See also \cite{tadmor2004mir}, \cite{aujol2006structure}, \cite{vixie2010multiscale}, \cite{le2010inter} for different approaches in obtaining the global scales from the point of view of image decompositions.

Realizing that $x$ may be embedded in multiple scales, Jones and the author \cite{jones2009local} propose a new method for extracting a vector of scales at each location $x$. This notion of local scales is further characterized based on the visibility level of the scales and their separation from other scales. This multiscale analysis is intimately related to the theory of wavelets \cite{daubechies1992ten} and square functions applied to $f$ (see Section \ref{ls_images} below). Here, we would like to replace the function $f$ with a subset $\Gamma\subset \R^n$, which we can think of as curves or surfaces. We note that in shape analysis, the knowledge of local scales on curves and surfaces is useful for shape matching and comparison. Just as in images \cite{lowe2004distinctive}, the local scales on $\Gamma$ at $x$ can be used to build distinctive features on $\Gamma$. Multiscale and wavelets analysis on curves and surfaces has been a wide subject of study (\cite{jones1990rectifiable}, \cite{david1991wavelets}, \cite{david1993analysis}, among others). Methods of denoising and reconstructing parametric curves are proposed by L-M  Reissell \cite{reissell1996wavelet} using wavelets, and recently by M. Feiszli and P. Jones \cite{feiszli2010sin} to denoise piecewise smooth curves while preserving sigularity. See also P.L. Rosin \cite{rosin1998determining} for a study of local scales on curves.

In this paper, we would like to extend the theory from \cite{jones2009local} to curves and surfaces. Our motivation comes form the work of P. Jones \cite{jones1990rectifiable}, and G. David and S. Semmes \cite{jones1990rectifiable}, where the multiscale analysis and analysis of square functions is applied to the set $\Gamma$ (instead of a function $f$). The paper is organized as follows. In Section \ref{ls_images}, we recall the study of local scales of a function from \cite{jones2009local}. In Section \ref{ls_surface}, we show that under the assumption that $\Gamma$ is a Lipschitz graph or $\Gamma$ contains big pieces of Lipschitz graph \cite{david1993analysis}, the analogous results \cite{jones2009local} can be extended to curves and surfaces. In the last section, we discuss possible techniques for diffusing curves and surfaces.

\section{Local scales in images}\label{ls_images}
We recall the study of local scales from \cite{jones2009local}. Let $\phi(x)$ be an approximation to the identity. In particular, let
$$
\phi(x) = e^{-\pi |x|^2}.
$$
Here we have $\int_{\R^n} \phi(x)\ dx = 1$. Another desirable choice of $\phi$ is the Poisson kernel, but here we will consider the $\phi$ given above. For each $t>0$, define
$$
K_t(x) = t^{-n/2}\phi(x/\sqrt{t}) = (t)^{-n/2} e^{-\pi\frac{|x|^2}{t}}.
$$
Note that $K_t(x)$ is the the Gaussian kernel in $\R^n$ satisfying
\begin{equation}
\d_t K_t(x) = (4\pi)^{-1}\Delta K_t(x).
\end{equation}
Moreover, the Fourier transform of $K_t(x)$ is given by \cite{stein1970}
$$
\widehat{K_t}(\xi) = \int_{\R^n} K_t(x) e^{2\pi i x\xi}\ dx = \int_{\R^n} (4\pi t)^{-n/2} e^{-\frac{|x|^2}{4t}}e^{2\pi i x\xi}\ dx=e^{-\pi t|\xi|^2},
$$
and for all $k\ge 0$,
\begin{equation}\label{sm_Kt}
\left\|\frac{\d^k}{\d t^k}K_t \right\|_{L^1(\R^n)}\le \frac{C_k}{t^k}.
\end{equation}

For each $t>0$, let $\psi_t(x) = t\frac{\d K_t}{\d t}$, $x\in\R^n$. From \eqref{sm_Kt}, we see that for all $k\ge 0$,
\begin{equation}\label{sm_psi}
\left\|\frac{\d^k}{\d t^k}\psi_t \right\|_{L^1(\R^n)}\le \frac{C_{k+1}}{t^k}.
\end{equation}

Let $f\in L^\infty(\R^n)$ be a given image. For each $x\in \R^n$ and $t>0$, define
\begin{equation}\label{ut}
Sf(x,t) =\psi_t*f(x) = \int_{\R^n} \psi_t(x-y)f(y)\ dy.
\end{equation}
Since $\psi_t$ has zero mean and zero first moments, we see that if $f$ is linear, then $Sf(x,t) = 0$ for all $t>0$ and $x\in \R^n$.  Thus the quantity $|Sf(x,t)|$ measures how $f$ deviates from a linear function of scale $t$ near $x$. Note that $Sf$ is invariant under addition by a linear function. 

\begin{remark}(Interpretation of $\psi_t$) Let $u(t) = K_t*f$ be the multiscale representation of $f$ with respect to the diffusion $K_t$. Then 
$$
Sf(x,t) = \psi_t*f = t\frac{\d}{\d t}u(t) = \ln(a)\frac{\d}{\d\tau}u(a^\tau),
$$ 
where $\tau = \log_a(t)$ for some $a>1$. $\frac{\d}{\d\tau}u(a^\tau)$ gives the change of $u$ in logarithmic scale. A big value of  $|\frac{\d}{\d\tau}u|$ (local max of $\frac{\d}{\d\tau}u(a^\tau)$) implies that there is a change in scale (logarithmic scale). One can also view $\psi_t$ as a wavelet, and $Sf(x,t)$ is the wavelet coefficient of $f$ at scale $t$ and location $x$. For each $x$, the family $\{u(x,t_i(x)): t_i\in T_f(x)\}$ gives a refined multiscale representation of $f$ at $x$.
\end{remark}

By \eqref{sm_psi}, we have for all $k\ge 0$,
\begin{equation}
\sup_{x\in \R^n} \left|\frac{\d^k}{\d t^k}Sf(x,t) \right| \le \frac{C_k}{t^k}\|f\|_{L^\infty}.
\end{equation}
By a change of variable, let $\tau = \log_a(t)$ for some fixed $a>1$, and let 
$$
Sf(x,\tau) := Sf(x,t),\mbox{ such that } \tau = \log_a(t).
$$
Then with respect to $\tau$, we have
\begin{equation}\label{sm_sf}
\sup_{x\in \R^n} \left|\frac{\d^k}{\d \tau^k}Sf(x,\tau) \right| \le C_k\|f\|_{L^\infty}.
\end{equation}

\begin{example}\label{ex_sin}
Let $f(x) = \sin(2\pi mx)$, $x\in \R$. A simple calculation shows that
$$
Sf(x,t) = t\frac{\d}{\d t}\left(2\sin(2\pi mx)e^{-\pi t|m|^2}\right) = -2\pi |m|^2 \sin(2\pi mx) te^{-\pi t|m|^2}.
$$
For each $x\in \R$ such that $\sin (2\pi mx)\neq 0$, $|Sf(x,t)|$ has only one local maxima at $t = \frac{1}{\pi m^2}$. 
\end{example}

\begin{definition} Let $f\in L^\infty(\R^n)$. For each $x\in \R^n$, the local scales of $f$ at $x$ is defined as the set
\begin{equation}\label{ls_f}
T_f(x) = \{t>0: |Sf(x,t)|\mbox{ is  a local maximum}\}.
\end{equation}
By a change of variable, let $\tau = \log_a(t)$ and $Sf(x,\tau) := Sf(x,a^\tau) = Sf(x,t)$. Denote
$$
\T_f(x) = \{\tau\in \R: t=a^\tau\in T_f(x)\}.
$$
\begin{itemize}
\item For each $\beta>0$, we say $\tau\in \T_f(x)$ is $\beta$-visible if $|Sf(x,\tau)| >\beta$.
\item For each $\delta>0$, we say $\tau\in \T_f(x)$ is $\delta$-separated if $|\frac{\d^2}{\d \tau^2}Sf(x,\tau)| >\delta$.
\end{itemize}
\end{definition}

For each $\delta >0$, denote by $d_\delta f(x) = f(\delta x)$ the dilating operator. The following property relates local scales of $f$ with its dilation. 

\begin{definition}
({\bf Dilating Consistency Property}) We say the set of local scales $T_f(x)$ satisfies the dilating consistency property if
$$
T_{d_\delta f}(\delta x) = \{\delta^s t: t\in T_f(x)\}, \mbox{ for some } s\in \R.
$$
\end{definition}
\begin{remark}
The Dilating Consistency property of local scales also implies that the number (cardinality) of local scales are invariant under dilation. In other words, we do not introduce or remove local scales as a result of dilating (zoom-in or zoom-out) $f$. The local scales of $f$ defined in \cite{brox2004tfb} and \cite{luo2006characteristic}, for instance, do not satisfy this condition. See \cite{le2010local} for an overview of the study of local scales in images.
\end{remark}

\begin{proposition}
The set of local scales defined by \eqref{ls_f} satisfies the Dilating Consistency property with $s=-2$.
\end{proposition}

We have the following characterization of $\beta$-visible and $\delta$-separated local scales. Let $\Omega\subset \R^n$ be a bounded domain. For each $x\in\Omega$ and $\delta>0$, let
$$
\tau_\delta(x) = \left\{\tau\in \T_f(x): \left|\frac{\d^2}{\d \tau^2}Sf(x,\tau)\right| >\delta\right\},
$$
and
\begin{equation}\label{omega_delta_N}
\Omega_{\delta,N} = \{x\in \Omega: \#\tau_\delta(x) \ge N\}.
\end{equation}
The following Corollary provides a characterization of $|\Omega_{\delta,N}|$ and it is an application to the John-Nirenberg Theorem \cite{john1961functions}. It tells us that the set of points which are embedded in many scales is small, and decays exponentially as a function of the number of scales.
\begin{corollary}\label{cor_delta_ls}
Suppose $f\in L^\infty(\R^n)$ and $\Omega\subset \R^n$ be bounded. Let $\Omega_{\delta,N}$ be defined as in \eqref{omega_delta_N}. Then there exist constants $C_1$ and $C_2$ which depend on $\|f\|_{L^\infty}$ and $|\Omega|$ such that
$$
|\Omega_{\delta,N}| \le C_1e^{-C_2N\delta^3}.
$$
\end{corollary}

The proof of this corollary \cite{jones2009local} essentially uses the fact that for any integer $k\ge 1$, the Littlewood-Paley $g_k$-function (a vector-valued of singular integral) defined as
$$
g_k(f)(x) = \left[\int_0^\infty\left|t^k\frac{\d^k K_t}{\d t^k}*f(x) \right|^2\frac{dt}{t} \right]^{1/2}
$$
is bounded in $L^2(\R^n)$ (see Chapter IV, Section 1 in \cite{stein1970}). I.e.
$$
\|g_k(f)\|_{L^2} \le A_2\|f\|_{L^2}.
$$

For $x\in \Omega$ and $\beta,\delta >0$, define
$$
\tau_{\beta,\delta}(x) = \left\{\tau\in \T_f(x): |Sf(x,\tau)|>\beta, \left|\frac{\d^2}{\d \tau^2}Sf(x,\tau)\right| >\delta\right\},
$$
and 
\begin{equation}\label{omega_beta_delta_N}
\Omega_{\beta,\delta,N} = \{x\in \Omega: \#\tau_{\beta,\delta}(x) \ge N\}.
\end{equation}
A similar result as in Corollary \ref{cor_delta_ls} also holds for $|\Omega_{\beta,\delta,N}|$.
\begin{corollary}\label{cor_beta_delta_ls}
Suppose $f\in L^\infty(\R^n)$ and $\Omega\subset \R^n$ be bounded. Let $\Omega_{\beta,\delta,N}$ be defined as in \eqref{omega_beta_delta_N}. Then there exist constants $C_1$ and $C_2$ which depend on $\|f\|_{L^\infty}$ and $|\Omega|$ such that
$$
|\Omega_{\beta,\delta,N}| \le C_1e^{-C_2N\delta^2\alpha},
$$
where $\alpha = \min(\beta,\delta)$.
\end{corollary}

For each $x\in \R^n$ and $t>0$, define the cone $\Gamma(x,t)$ as
\begin{equation}\label{cone_eq}
C(x,t) = \{y\in \R^n: \pi|x-y|^2<t\}.
\end{equation}
Let the non-tangential control $S^*f(x,t)$ of $Sf(x,t)$ be defined as
\begin{equation}\label{s*f}
S^*f(x,t) = \sup_{y\in C(x,t)} |Sf(y,t)|e^{-\pi\frac{|x-y|^2}{t}},
\end{equation}
which is bounded by $C_1\|f\|_{L^\infty}$.

Recall from example \ref{ex_sin} that for $f(x) =\sin (2\pi mx)$, $|Sf(x,t)|$ is given by
$$
|Sf(x,t)| =  |\sin(2\pi mx)| (2\pi |m|^2t)e^{-\pi t|m|^2}.
$$
Thus the visibility of the scale $t = 1/(\pi m^2)$ at $x$ depends on the value of $|\sin(2\pi mx)|$. We would like to think that for all $x\in \R$, the local scale at $x$ should have the same visibility. In other words, suppose $x$ has a local scale $t$. Then for all $y$ such that $|x-y|< O(\sqrt{t})$, we should expect that $y$ also has the same local scale as $x$. This was the motivation to consider the non-tangential function $S^*f(x,t)$ as an approach to lift the visibility level locally. To improve the $\beta$-visible local scales, we consider the following non-tangential local scales \cite{jones2009local}.
\begin{definition}
({\bf Non-tangential local scales}) The non-tangential local scales of $f$ at $x$ is defined as the set 
$$
T^*_f(x) = \{t>0: S^*(x,t) \mbox{ is a local maximum}\}.
$$
\end{definition}
With the appropriate choice of the cone $C(x,t)$, in particular the one defined in \ref{cone_eq}, the non-tangential local scales also satisfy the Dilating Consistency Property.

\section{Local scales on curves and surfaces}\label{ls_surface}
In this section, we are interested in applying the study of local scales of functions defined on $\R^n$ from Section \ref{ls_images} to studying local scales on $\Gamma$, a $d$-dimensional subset of $\R^n$. In the usual sense, if $d=1$, then $\Gamma$ is a curve in $\R^n$ and for all other cases $1<d< n$, $\Gamma$ is a surface in $\R^n$. However, to fix notation, we call $\Gamma$ the $d$-dimensional surface for all $1\le d<n$.
We would like to have the counter parts of the analysis of singular integral operators on functions in a geometrical setting. Fortunately, this geometrical setting has been studied to a great extend since the early 1980's (\cite{coifman1982intŽgrale}, \cite{jones1988lipschitz}, \cite{murai1986boundedness}, \cite{christ1987polynomial}, among others). However, our motivation is from the work of P.W. Jones \cite{jones1990rectifiable}, and G. David and S. Semmes \cite{david1991harmonic}.

We begin with defining different notions of regularity and rectifiability on the surface $\Gamma$ \cite{david1993analysis}.
\begin{definition}
Let $\Gamma\subset \R^n$ with Hausdorff dimension $d$. 
\begin{enumerate}
\item We say $\Gamma$ is a $d$-dimensional {\bf Lipschitz graph} (with constant $C$) if there is a $d$-plane $P$, and $(n-d)$-plane $P^\perp$ orthogonal to $P$, and a Lipschitz function $A:P\rightarrow P^\perp$ (with norm $C$) such that
$$
\Gamma = \{p+A(p): p\in P\}.
$$
By a change of coordinate system, we can view $P\subset \R^d$ and write $\Gamma$ as
$$
\Gamma = \{(p,A(p)): p\in P\}.
$$
\item We say $\Gamma$ is (countably) {\bf rectifiable} if there is a countable family $A_j$ of Lipschitz maps from $\R^d$ to $\R^n$ such that
$$
\H^d\left(\Gamma\setminus \left(\cup_{j} A_j(\R^d)\right)\right) = 0.
$$
Here $\H^d$ denotes the $d$-dimensional Hausdorff measure.
\item We say $\Gamma$ is {\bf regular} if it is closed and if there exists a constant $C$ such that
$$
C^{-1}r^d \le \H^d(\Gamma\cap B(x,r)) \le C r^d,
$$
for all $x\in \Gamma$ and $r>0$.
\item We say $\Gamma$ has {\bf BPLG} (big pieces of Lipschitz graphs) if it is regular and if there exist $C$, $\epsilon>0$ so that for every $x\in \Gamma$ and $r>0$, there is a $d$-dimensional Lipschitz graph  $E$ (with constant $\le C$) such that
$$
\H^d(\Gamma\cap B(x,r)\cap E) \ge \epsilon r^d.
$$
\end{enumerate}
\end{definition}
From the definition, BPLG implies rectifiability. 

In \cite{jones1990rectifiable}, P. Jones provides a geometric solution to the traveling salesman problem. In particular, let $K\subset \C$ be a bounded set. P. Jones gives a necessary and sufficient condition for $K$ to lie in a rectifiable curve using an $L^\infty$-type beta number. Let $Q\subset \C$ be a dyadic square given by
$$
Q = [j2^{-n},(j+1)2^{-n}]\times [k2^{-n}, (k+1)2^{-n}],\mbox{ where } j,k,n\in \Z.
$$
Denote by $l(Q) = 2^{-n}$ the side length of $Q$. For $\lambda>0$, denote by $\lambda Q$ the square with same center as $Q$, with side length $\lambda l(Q)$, and sides parallel to the axes. For each dyadic $Q$, let $S_Q$ be an infinite strip of smallest possible width ($S_Q$ could be a line) which contains $K\cap 3Q$, and let $w(Q)$ denote the width of $S_Q$. Define
$$
\beta(Q) = \frac{w(Q)}{l(Q)},
$$ 
which is scale invariant, and measures the deviation of $K$ from a straight line ($1$-plane) near $Q$ at scale $l(Q)$. P. Jones \cite{jones1990rectifiable} then shows that $K$ is contained in a rectifiable curve $\Gamma$ if and only if
$$
\beta^2(K) := \sum_{Q}\beta^2(Q)l(Q) < \infty.
$$
In particular, if $\Gamma$ is connected, then 
$$
\beta^2(\Gamma) := \sum_{Q}\beta^2(Q)l(Q) < C l(\Gamma).
$$

In \cite{david1991harmonic}, David-Semmes provide various equivalent geometric quantities for $d$-dimensional subsets in $\R^n$. Let $\Gamma\subset \R^n$ be a closed $d$-dimensional subset, and denote by $\mu$ the $d$-dimensional Hausdorff measure restricted to $\Gamma$. Let $\psi$ be a good kernel, in particular an odd, smooth, and compactly supported function defined on $\R^n$. Denote by $\psi_t(x) = t^{-d}\psi(x/t)$. For each $t>0$, define
\begin{equation}\label{ds_psi}
\psi_t*\mu(x) = \int_\Gamma \psi_t(x-y)\ d\mu(y).
\end{equation}
$\psi_t$ being odd implies that $\psi_t*\mu = 0$ if $\Gamma$ is a $d$-plane. Thus, the quantity $|\psi_t*\mu(x)|$ also measures the deviation of $\Gamma$ from a $d$-plane near a neighborhood of scale $t$ at $x$. In \cite{david1991harmonic}, David-Semmes show that the quantity $\psi_t*\mu(x)$, $t = 2^{-m}$, is intimately related to the $L^1$ version of Jones's beta number $\beta(Q)$, where $x\in Q$ and $l(Q) = 2^{-m}$. More specifically, define
\begin{equation}\label{beta_1}
\beta_1(x,t) = \inf_{P} t^{-d}\int_{\Gamma\cap B(x,t)} \frac{dist(y,P)}{t}\ d\mu(y),
\end{equation}
where the infimum is taken over all $d$-planes $P$. Then the following two conditions are equivalent.
\begin{itemize}
\item[(P1)] Denote by $d\delta_{2^{-m}}(t)$ the Dirac mass in $t$ at $2^{-m}$. Then
$$
\sum_{m=-\infty}^\infty \left|\psi_t*\mu(x) \right|^2\ d\mu(x)d\delta_{2^{-m}}(t)
$$
is a Carleson Measure on $\Gamma\times \R^+$.
\item[(P2)] $\beta_1(x,t)^2\ d\mu(x)\frac{dt}{t}$ is a Carleson measure on $\Gamma\times \R^+$.
\end{itemize}

In this paper, we would like to replace $\psi_t$ in \eqref{ds_psi} with one that is analogous to \eqref{ut} which is smooth, symmetric and decays exponentially, with the additional properties that $\psi_t$ has zero mean and zero first moments. The quantity $\psi_t*\mu(x)$ is then used to study local scales on $\Gamma$ (compare this with $\psi_t*f(x)$ in \eqref{ut}). This can be seen as an extension of the study of local scales from \cite{jones2009local} to curves and surfaces.

Let $P_d = \{(x_1,\cdots, x_n)\in \R^n: x_i=0, \forall i>d\}\subset \R^n$ be an affine $d$-plane, and for $x\in \R^n$, let $\phi(x) = e^{-\pi |x|^2}$. We have $\int_{P_d}\phi(x)\ d\H^d(x) = 1$. Note that since $\phi$ is radially symmetric, we have that for any $P$ which is a rotation of $P_d$ at the origin,
$$
\int_P \phi(x)\ d\H^d(x) = 1.
$$ 
For each $t>0$, define
$$
K_t(x) = t^{-d/2}\phi\left(\frac{x}{\sqrt{t}}\right).
$$
Then we have
$$
\int_{P_d}K_t(x)\ d\H^d(x) = 1,\mbox{ for all } t>0.
$$
Moreover, for all $k\ge 1$,
\begin{equation}\label{Kt_dt}
\int_{P_d}\left|\frac{\d^k}{\d t^k}K_t(x)\right|\ d\H^d(x)\le \frac{c_k}{t^k}.
\end{equation}

Define $\psi_t (x)= t\frac{\d}{\d t}K_t(x)$ for $x\in \R^n$ and $t>0$, then it is easy to show that 
$$
\int_{P_d} \psi_t(x)\ d\H^d(x) = 0,
$$ 
and $\psi_t$ also has zero first moments. Now, let $\Gamma\subset \R^n$ be a $d$-dimensional subset, and let $\mu$ be the $d$-dimensional surface or Hausdorff measure restricted to $\Gamma$. For each $x\in \Gamma$ and $t>0$, define
\begin{equation}\label{SGamma}
S\Gamma(x,t) = \psi_t*\mu(x) := \int_\Gamma \psi_t(x-y)\ d\mu(y).
\end{equation}
By the property of $\psi_t$, we see that if $\Gamma$ is an affine  $d$-plane, then $S\Gamma(x,t) = 0$ for all $t>0$. Locally, the quantity $|S\Gamma(x,t)|$ measures how well $\Gamma$ is approximated by an affine $d$-plane near $x\in \Gamma$ at scale $t$. 

%

\begin{remark}\label{inv_r_t}
Since $\psi_t$ is radially symmetric, $S\Gamma(x,t)$ is invariant under rotation, translation. In other words, let $A$ be a transformation from $\R^n$ to $\R^n$ consisting of a rotation and a translation, and denote by $\Gamma_A = \{Ax:x\in \Gamma\}$. Then
$$
S\Gamma_A(Ax,t) = S\Gamma(x,t).
$$
Thus, if we define local scales using $S\Gamma(x,t)$, then the local scales are also invariant under rotation and translation.
\end{remark}

Consider the following example where $\Gamma$ is a Lipschitz curve in $\R^2$.
\begin{example}\label{1d_ex}
Let $I=[0,1]$ be an interval  in $\R$ and $A:I\rightarrow \R$ be bounded and Lipschitz. Let $\Gamma = \{f(r) =(r,A(r)): r\in I\}\subset \R^2$ with $\mu$ being the length measure on $\Gamma$. Then for a fixed $x = f(r)\in \Gamma$ we have
$$
S\Gamma(x,t) = \int_\Gamma \psi_t(x-y)\ d\mu(y) = \int_I \psi_t(f(r) -f(s)) |f'|\ ds,
$$
where $|f'(r)| = \sqrt{1+|A'(s)|^2}$. Since $\psi_t$ is radially symmetric, we have
$$
S\Gamma(f(r),t) = \int_I \psi_t(\|f(r) - f(s)\|) |f'|\ ds,
$$
Suppose $A(s) = \sin(ns)$, then $A'(s)= n\cos(ns)$, which shows that $|f'|\approx n$ for large n, and hence $S\Gamma(x,t)\approx n\int_I \psi_t(\|f(r) - f(s)\|)\ ds$. In another case, suppose the graph of $A$ consists of tents such that $|A'(r)| = n$ a.e.. Then in this case, $S\Gamma(x,t)=  (1+n^2)^{1/2}\int_I \psi_t(\|f(r) - f(s)\|)\ ds$ a.e., which is mainly determined by 
$$
\psi_t*\alpha(x) = \int_I \psi_t(\|f(r) - f(s)\|)\ ds, 
$$
where $\alpha$ is the 1-dimensional Hausdorff measure ($d\alpha = ds$).
\end{example}

To extend the results from \cite{jones2009local}, we are interested in the following questions.
\begin{question}\label{q1}
What are the necessary conditions on $\Gamma$ so that for all $k\ge 0$,
\begin{equation}\label{q1_eq1}
\sup_{x\in \Gamma}\left|\frac{\d^k}{\d t^k}\psi_t*\mu(x)\right|\le \frac{C_{k,\Gamma}}{t^k},
\end{equation}
where $C_{k,\Gamma}$ depends on $k$ and $\Gamma$.
\end{question}

\begin{question}
What are the necessary conditions on $\Gamma$ so that the $g_k$ function defined by
\begin{equation}\label{lp_gk_gamma}
g_k(f)(x) = \left[\int_{0}^\infty \left|t^k\frac{\d^k}{\d t^k}\psi_t*f(x)\right|^2\ \frac{dt}{t}\right]^{1/2},
\end{equation}
is bounded in $L^2(\Gamma)$. In particular, we would like to know if the following inequality holds.
\begin{equation}\label{q2_eq1}
\left[\int_{\Gamma} |g_k(f)(x)|^2\ d\mu(x)\right]^{1/2}\le C \|f\|_{L^2(\Gamma)}.
\end{equation}
In particular, take $f$ to be the characteristic function of $\Gamma\cap B_r$, then the above condition implies $\left[t^k\frac{\d^k}{\d t^k}\left(\psi_t*\mu(x)\right)\right] \frac{dt}{t}d\mu(x)$ is a Carleson measure.
\end{question}


We begin by considering $\Gamma = \{z(r) = (r,A(r)): r\in F\}$ to be a $d$-dimensional Lipschitz graph, for some closed subset set $F$ in $\R^d$. In this situation, we see that the surface measure $\mu$ and the Hausdorff measure $\alpha$ are equivalent. Let $E\subset \Gamma$. Denote by
$$
\|z'(r)\| = \sqrt{det[(g_{i,j})_{i,j=1,\cdots,d} ]},\mbox{ where } g_{ij} =\left \langle \frac{\d z}{\d r_i}, \frac{\d z}{\d r_j}\right\rangle,
$$
and
\begin{equation}\label{gamma_star}
\|\Gamma\|_* = \sup_{r\in F} \|z'(r)\|.
\end{equation}
We have
$$
\mu(E) = \int_\Gamma \chi_E(x)\ d\mu(x) = \int_{F} \chi_E(z(r)) \|z'(r)\|\ dr,
$$
and
$$
\alpha(E) = \int_{F}\chi_E(z(r))\ dr.
$$
But $1\le \|z'(r)\| \le \|z\|_{lip} = C$. This implies
$$
\alpha(E) \le \mu(E)\le C\alpha(E).
$$
Thus it is equivalent to consider either $\mu$ or $\alpha$ on $\Gamma$.

\begin{proposition}\label{prop_gamma_1}
Let $\Gamma$ be a $d$-dimensional Lipschitz graph with the $d$-dimensional surface measure $\mu$. Then for all $k\ge 0$, we have
\begin{equation}\label{Sg_lips}
\sup_{x\in \Gamma}\left|\frac{\d^k}{\d t^k}S\Gamma(x,t)\right|\le \frac{C_{k}}{t^k}\|\Gamma\|_{*},
\end{equation}
where $\Gamma$ is defined in \eqref{gamma_star} and $C_k$ depends only on $k$ and $d$. 
\end{proposition}

\begin{remark}
Note that given any closed subset $F$ in $\R^d$ and a Lipschitz function $A:F\rightarrow \R^{n-d}$, the function $A$ can be extended to $\R^n$ with the same Lipschitz constant using an extension theorem of Whitney (see Theorem 3 of Chapter VI, Section 2 of A. Stein \cite{stein1970}). Thus without lost of generality, we may assume
$$
\Gamma = \{(x,A(x)): x\in \R^d, \|A\|_{Lip} <\infty\}.
$$
\end{remark}

\begin{remark}\label{remark_bound1}
Let $\Gamma$ be defined as in Proposition \ref{prop_gamma_1} with $F=\R^d$. Fix an $x_0 = (r_0,A(r_0))\in\Gamma$. For all $k\ge 0$, let 
\begin{equation}\label{psi_tk}
\psi_{t,k}(x_0-x) = c_d t^{-d/2}\left[\pi \frac{|x_0-x|^2}{t}\right]^k e^{-\pi\frac{|x_0-x|^2}{t}},\ x = (r,A(r))\in\Gamma.
\end{equation}
The first estimate is the following.
\begin{equation}\label{psi_bound}
\int_{\R^d}\psi_{t,k}(x_0-x)\ dr \le C_k,
\end{equation}
where $C_k$ is a constant that depends only on $k$ and $d$. Indeed,
\begin{equation}\label{est_1}
\begin{split}
&\int_{\R^d}\psi_{t,k}(x_0-x)\ dr = c_d\int_{\R^d} t^{-d/2}\left[\pi\frac{|r-r_0|^2 + |A(r)-A(r_0)|^2}{t}\right]^k e^{-\pi\frac{|r-r_0|^2 + |A(r)-A(r_0)|^2}{t}}\ dr \\
&=c_d \int_{\R^d}\left[\pi(|s|^2 + |\frac{A(\sqrt{t}s) - A(r_0)}{\sqrt{t}}|^2)\right]^ke^{-\pi(|s|^2 + |\frac{A(\sqrt{t}s) - A(r_0)}{\sqrt{t}}|^2)}\ ds\\
&=c_d\int_{S^d}\int_0^\infty \left[\pi(\gamma^2 + |\frac{A(\sqrt{t}\gamma) - A(r_0)}{\sqrt{t}}|^2)\right]^ke^{-\pi(\gamma^2 + |\frac{A(\sqrt{t}\gamma) - A(r_0)}{\sqrt{t}}|^2)}\ \gamma\ \d\gamma\ \d\theta\\
&\le c_d\omega_d\int_0^\infty \left[\pi(\gamma^2 + |\frac{A(\sqrt{t}\gamma) - A(r_0)}{\sqrt{t}}|^2)\right]^{k+1/2}e^{-\pi(\gamma^2 + |\frac{A(\sqrt{t}\gamma) - A(r_0)}{\sqrt{t}}|^2)}\  \d\gamma\ \d\theta\\
&\le c_d\omega_d\int_{\gamma<\sqrt{\frac{k+1/2}{\pi}}} (k+1/2)^{k+1/2}e^{-{(k+1/2)}}\ ds + \int_{\gamma\ge\sqrt{\frac{k+1/2}{\pi}}} (\pi\gamma^2)^{k+1/2} e^{-\pi \gamma^2}\ d\gamma\\
&\le (k+1/2)^{(k+1/2)}e^{-(k+1/2)}\omega_d\frac{\sqrt{k+1/2}}{\pi} + c_k = C_k<\infty.
\end{split}
\end{equation}
The constant $[\frac{p}{e}]^p$ is large for $p\gg e$. Note that $(\pi \gamma^2)^pe^{-\pi\gamma^2}$ achieves its maximum when $\gamma = \sqrt{\frac{p}{\pi}}$, and 
$$
p^pe^{-p} = \sup_{\gamma>0} \{(\pi \gamma^2)^pe^{-\pi\gamma^2}\}.
$$
Denote by 
$$
\|A'\|_{L^\infty} := \|\nabla A\|_{L^\infty} = \sup_{r\in\R^d} \left[\sum_{i=1}^d\left|\frac{\d A}{\d r_i}(r) \right|^2 \right]^{1/2}.
$$ 
The following is another estimate for $\int_{\R^d}\psi_{t,k}(x_0-x)\ dr$.
\begin{equation}\label{est_2}
\int_{\R^d}\psi_{t,k}(x_0-x)\ dr \le \int_{\R^d} t^{-d/2}(1+\|A'\|_{L^\infty}^{2})^k\left[\pi\frac{|r-r_0|^2}{t}\right]^k e^{-\pi\frac{|r-r_0|^2}{t}}\ dr = A_k(1+\|A'\|_{L^\infty}^{2})^k,
\end{equation}
where
$$
A_k = c_d\int_{\R^d} (\pi|r|)^{2k}e^{-\pi |r|^2}\ dr.
$$
Note that the first estimate \eqref{est_1} does not depend on $\|A'\|_{L^\infty}$, while the second estimate \eqref{est_2} does. Thus we see that if $\|A'\|_{L^\infty}$ is large then \eqref{est_1} is a better estimate than \eqref{est_2}.
\end{remark}

\begin{proof} (Proof of Proposition \ref{prop_gamma_1})
From remark \ref{inv_r_t}, we may assume $P=\R^d$, and $P^\perp = \R^{n-d}$. Without lost of generality, we may assume
$$
\Gamma = \{(x,A(x)): x\in \R^d, \|A\|_{Lip} <\infty\}.
$$
Fix an $x_0\in \Gamma$. Observe that for all $x\in \Gamma$,
$$
\frac{\d^k}{\d t^k} \psi_t(x_0-x) = \frac{1}{t^k} \sum_{i=0}^{k+1} c_i \psi_{t,i}(x_0-x),
$$
where $\psi_{t,i}(x_0-x)$ is defined in \eqref{psi_tk}, and $c_i$ are constants that depend only on $d$. This implies
\begin{equation}
\begin{split}
\left|\frac{\d^k}{\d t^k}S\Gamma(x_0,t)\right| &\le \int_\Gamma \left|\frac{\d^k}{\d t^k}\psi_t(x_0-x)\right|\ d\mu(x) = \int_{\R^d} \left|\frac{\d^k}{\d t^k}\psi_t((r_0,A(r_0))-(r,A(r)))\right| \|\Gamma\|_{*}\ dr\\
&= \frac{\|\Gamma\|_{*}}{t^k} \sum_{i=0}^{k+1} |c_i|\int_{\R^d} \psi_{t,i}((r_0,A(r_0))-(r,A(r)))\ dr\le \frac{\|\Gamma\|_{*}}{t^k}\sum_{i=0}^{k+1} |c_i|C_i,
\end{split}
\end{equation}
where the last inequality follows from \eqref{est_1} in remark \ref{remark_bound1}. Thus,
$$
\left|\frac{\d^k}{\d t^k}S\Gamma(x_0,t)\right| \le \frac{C_k}{t^k}\|\Gamma\|_{*},
$$
where $C_k$ is a new constant that depends only on $d$ and $k$. Since $x_0\in\Gamma$ is arbitrary, we have that \eqref{Sg_lips} holds.
\end{proof}

Let $z:\R^d\rightarrow \R^{n}$ be a continuous function (not necessarily Lipschitz or differentiable), and suppose 
$$
\Gamma = \{f(r): r\in \R^d\}.
$$
For $x=z(r)$, let $d\alpha(x) = dr$, and define
$$
S\Gamma(x,t) = \psi_t*\alpha(x) = \int_{\R^d} \psi_t(z(r) - z(s))\ ds.
$$ 
Then as a consequence to Proposition \ref{prop_gamma_1}, we have
\begin{equation}\label{sm_sg_haus}
\left|\frac{\d^k}{\d t^k}S\Gamma(x_0,t)\right| \le \frac{C_k}{t^k},
\end{equation}
where $C_k$ does not depend on $\Gamma$. Note that in this case, we get the same bound as in \eqref{Sg_lips} but the quantity $\|\Gamma\|_*$ is removed. Thus by restricting to a $d$-dimensional Hausdorff measure, the condition on $f$ can be weakened.

Next, we would like to address Question 2, in particular, the condition \eqref{q2_eq1}. To this end, we follow  \cite{david1991wavelets}.

\begin{proposition}(Part II, Section 6, Example 6.7 in \cite{david1991wavelets})\label{david1991}
Let $0<d\le n$ be integers, and let $k(x)$ be a $C^\infty$ function, defined on $\R^n\setminus {0}$, and such that
\begin{equation}\label{sio_eq1}
|\nabla^jk(x)|\le C(j) |x|^{-d-j} \mbox{ for all } j\ge 0,
\end{equation}
and
\begin{equation}\label{sio_eq2}
\sup_{0<\epsilon<M}\left|\int_{\epsilon<|t|<M} k(t\theta)|t|^{d-1}\ dt \right|\le C \mbox{ for all } \theta\in S^{n-1}.
\end{equation}
Let $A:\R^d\rightarrow \R^{n-d}$ be a Lipschitz function. Then the kernel
$$
K(r,s) = k(r-s,A(r) - A(s))
$$
defines a bounded singular integral operator on $L^2(\R^d)$.
\end{proposition}

\begin{remark}
An equivalent condition to \eqref{sio_eq2} is the following
\begin{equation}\label{sio_eq3}
\sup_{\epsilon>0} \left|\int_{\{|t|>\epsilon\}}|t|^{-d/2}k(\theta/\sqrt{t})\frac{dt}{t} \right| \le C\mbox{ for all } \theta\in S^{n-1}.
\end{equation}
\end{remark}

Using the previous Proposition, we would like to show that \eqref{q2_eq1} holds.

\begin{corollary}\label{cor_gk}
Let $\Gamma = \{x=(r,A(r)): r\in \R^d\}$ be a $d$-dimensional Lipschitz graph, for some Lipschitz function $A:\R^d\rightarrow \R^{n-d}$. Then the Littlewood-Paley function $g_k$ defined in \eqref{lp_gk_gamma} is bounded in $L^2(\Gamma)$. In particular, let $x = (r,A(r))$
\begin{equation}\label{cor_gk_eq1}
\left[\int_{\R^d} |g_k(f)(x)|^2\ dr\right]^{1/2}\le C \left[\int_{\R^d}|f(x)|^2\ dr\right]^{1/2},
\end{equation}
\end{corollary}
\begin{proof}
We will show the case for $g_0$. The general case $g_k$ will follow using the same techniques. We have
$$
g_0(f)(x) =\left[\int_0^\infty \left|\psi_t*f(x) \right|^2\ \frac{dt}{t}\right]^{1/2} =  \left[\int_0^\infty \left|t\frac{\d}{\d t}K_t*f(x) \right|^2\ \frac{dt}{t}\right]^{1/2}.
$$
We will follow Chaper IV, section 1.3 from E. Stein \cite{stein1970} by verifying the conditions \eqref{sio_eq1} and \eqref{sio_eq3} in the context of Hilbert space-valued functions.

Let $\H$ be the $L^2$ space on $(0,\infty)$ with measure $t\ dt$, i.e.
$$
\H = \{f: |f|_{\H}^2 := \int_0^\infty |f(t)|^2t\ dt < \infty\}.
$$
Let $k(x) = \frac{\d K_t}{\d t}(x)$. We will show that $k(x)$ satisfies
\begin{equation}\label{sio_H_eq1}
|\nabla^jk(x)|_\H\le C(j) |x|^{-d-j} \mbox{ for all } j\ge 0,
\end{equation}
and
\begin{equation}\label{sio_H_eq2}
\sup_{\epsilon>0} \left|\int_{\{|t|>\epsilon\}}|t|^{-d/2}k(\theta/\sqrt{t})\frac{dt}{t} \right|_\H \le C\mbox{ for all } \theta\in S^{n-1}.
\end{equation}
The condition \eqref{sio_H_eq1} clearly holds for it is a direct computation of the integral
$$
\int_0^\infty \left|\nabla^jk(x) \right|^2t\ dt.
$$
See for instance Chaper IV, Section 1.3 from E. Stein \cite{stein1970} when $K_t$ is the poisson kernel. We will show here condition \eqref{sio_H_eq2}. Since the kernel $k(x)$ is symmetric, it suffices to show 
\begin{equation}\label{sio_H_eq3}
\sup_{\epsilon>0} \left|\int_{\{r>\epsilon\}}r^{-d/2}k(\theta/\sqrt{r})\frac{dr}{r} \right|_\H \le C\mbox{ for all } \theta\in S^{n-1}.
\end{equation}
For each $\epsilon>0$, we have
\begin{equation*}
\begin{split}
&\int_0^\infty  \left|\int_{\{r>\epsilon\}}r^{-d/2}k(\theta/\sqrt{r})\frac{dr}{r} \right|^2 t\ dt = \int_0^\infty \left|\int_{\{r>\epsilon\}}r^{-d/2}\frac{\d K_t}{\d t}(\theta/\sqrt{r})\frac{dr}{r} \right|^2 t\ dt\\
&= \int_0^\infty \left|\int_{\{r>\epsilon\}}r^{-d/2-1}t^{-d/2-1}\left[-(d/2) + \frac{\pi |\theta|^2}{t r}\right]e^{-\pi\frac{|\theta|^2}{tr}}\ dr \right|^2 t\ dt\\
&= \int_0^\infty \left|t^{-d/2-1} \int_\epsilon^\infty \frac{\d}{\d r}K_r(\theta/\sqrt{t})\ dr\right|^2 t\ dt = \int_0^\infty \left|t^{-d/2-1}K_\epsilon(\theta/\sqrt{t})\right|^2 t\ dt\\
&=  \int_0^\infty \left|t^{-d/2-1}\epsilon^{-d/2}e^{-\pi \frac{|\theta|^2}{t\epsilon}}\right|^2 t\ dt =(\pi |\theta|^2)^{-d} \int_0^\infty r^{d-1}e^{-2r}\ dr = C_d \pi^{-d},
\end{split}
\end{equation*}
where $C_d = \int_0^\infty r^{d-1}e^{-2r}\ dr <\infty$ with $1\le d\le n$. Thus \eqref{sio_H_eq3} holds with $C = \left[C_d\pi^{-d}\right]^{1/2}$.
\end{proof}

\begin{remark}
To show \eqref{cor_gk_eq1} for general $g_k$, it suffices to show that the function $h_i(f)$ defined by
$$
h_i(f)(x) = \left[\int_0^\infty \left|t^i\frac{\d^k}{\d t^i}K_t*f(x) \right|^2\ \frac{dt}{t}\right]^{1/2}
$$
is bounded on $L^2(\Gamma)$. In this case, the Hilbert space $\H$ in consideration is the $L^2$ space on $(0,\infty)$ with measure $t^{2k-1}\ dt$
\end{remark}

\begin{remark}
In Corollary \ref{cor_gk}, we consider $\Gamma$ to be a Lipschitz graph. However, as noted in \cite{david1991wavelets} (Part III), the kernel $k$ satisfying \eqref{sio_eq1} and \eqref{sio_eq2} also defines a bounded singular operator from $L^2(\Gamma)$ to $L^2(\Gamma)$ with $\Gamma$ satisfying a weaker constraint, in particular, if $\Gamma$ has BPLG. This shows that Corollary \ref{cor_gk} also holds for $\Gamma$ having BPLG.
\end{remark}

%
%

For the remaining part of the paper, we assume that $\Gamma\subset \R^n$ is a $d$-dimensional surface with $\mu$ being either the surface or Hausdorff measure such that for all $k\ge 0$,
\begin{equation}\label{sm_cond1}
\sup_{x\in \Gamma}\left|\frac{\d^k}{\d t^k}S\Gamma(x,t)\right|\le \frac{C_{k,\Gamma}}{t^k}.
\end{equation}
We will also assume that the Littlewood-Paley $g_k$ function is bounded in $L^2(\Gamma)$, i.e.
\begin{equation}\label{sm_cond2}
\|g_k(f)\|_{L^2(\Gamma)}  =\left[ \int_\Gamma \int_0^\infty \left|t^k\frac{\d^k K_t}{\d t^k}*f(x) \right|^2\frac{dt}{t}\ d\mu(x)\right]^{1/2}\le C \|f\|_{L^2(\Gamma)}
\end{equation}

\begin{definition}
Let $\Gamma$ be such that \eqref{sm_cond1} holds. For each $x\in \Gamma$, the set of local scales $T_\Gamma(x)$ of $\Gamma$ at $x$ is defined as
$$
T_\Gamma(x) = \{t\in (0,\infty): |S\Gamma(x,t)| \mbox{ is a local maximum}\}.
$$
\end{definition}

It is easy to show the above definition of local scales satisfy the dilating consistency property. In other words, for $\delta >0$, let $d_\delta \Gamma$ denote the dilated version of $\Gamma$, i.e.
$$
d_\delta\Gamma=\{\delta x: x\in \Gamma\}.
$$
Then the following result holds.
\begin{proposition} We have
\begin{equation}
T_{d_\delta\Gamma}(\delta x) = \{\delta^{-2} t: t\in T_\Gamma(x)\}.
\end{equation}
\end{proposition}

By a change of variable, let $\tau = \log_a t$ for some $a>1$. Then the condition in \eqref{sm_cond1} implies that
\begin{equation}\label{sm_cond_log}
\sup_{x\in \Gamma}\left|\frac{\d^k}{\d \tau^k}S\Gamma(x,\tau)\right|\le C_k,
\end{equation}
where this new constant $C_k$ also depends on $a$. Denote
$$
\mathcal{T}_\Gamma(x) = \{\tau = \log_a t: t\in T_\Gamma(x)\}.
$$
Recall that for each $\tau\in \mathcal{T}_\Gamma(x)$, $|S\Gamma(x,\tau)|$ measures how $\Gamma$ deviates from an affine $d$-plane locally at logrithmic scale $\tau$ near $x$. In other words, $|S\Gamma(x,\tau)|$ locally measures the visibility of curviness of $\Gamma$ at $x$. Moreover, for each $x\in \Gamma$, $S\Gamma(x,\tau)$ (as a function of $\tau$) has the $k^{th}$ derivative bounded by the constant $C_k$. This allows us to say something about the distribution of local scales $\mathcal{T}_\Gamma(x)$ at $x$. In particular, we have the following types of local scales.

\begin{definition}
Let $\Gamma\subset \R^n$ such that \eqref{sm_cond_log} holds.
\begin{itemize}
\item For each $\beta>0$. We say a $\tau\in \mathcal{T}_\Gamma(z)$ is $\beta$-visible if $|S\Gamma(z,\tau)|>\beta$. 
\item For each $\delta>0$. We say $\tau\in \mathcal{T}_\Gamma(z)$ is $\delta$-separated if $\left|\frac{\d^2}{\d\tau^2}S\Gamma(z,\tau)\right| >\delta$. 
\end{itemize}
\end{definition}

For each $\delta>0$ and $z\in \Gamma$, denote by
$$
\mathcal{T}_{\delta}(z) = \left\{\tau\in \mathcal{T}_\Gamma(z): \left|\frac{\d^2}{\d\tau^2}S\Gamma(z,\tau)\right| >\delta\right\}.
$$
For each $N>0$, denote by
$$
\Gamma_{\delta,N} = \left\{z\in \Gamma: \#\mathcal{T}_{\delta}(z) > N\right\}.
$$
Then the same result as in Corollary \ref{cor_delta_ls} also holds for $\Gamma$.

\begin{corollary} \label{cor_gamma_delta_n}
Let $\Gamma$ be a bounded $d$-dimensional subsets of $\R^n$ such that \eqref{sm_cond2} and \eqref{sm_cond_log} hold. Then there exist constants $C_1$ and $C_2$ (depending on $\Gamma$) such that 
$$
\mu(\Gamma_{\delta,N}) \le C_1e^{-C_2\delta^3N}.
$$
\end{corollary}
\begin{proof}
without lost of generality, we may assume $\mu(\Gamma) = 1$. The proof can be carried out in the exact same manner as in \cite{jones2009local}. For completeness, we show the steps here.
By a change of variable, $\tau = \log_a(t)$, let $S\Gamma(x,\tau) = S\Gamma(x,t)$. Then we have
$$
\frac{\d^2}{\d\tau^2}Sf(x,\tau) = \phi_t*\mu(x),
$$
where $\phi_t = (\ln(a))^2\left[t^2\frac{\d^2}{\d t^2}\psi_t + t\frac{\d}{\d t}\psi_t \right]$. Define the square function
$$
\mathcal{S}^2\Gamma(x) = \int_0^\infty |\phi_t*\mu(x)|^2\frac{dt}{t} = \ln(a)\int_{-\infty}^\infty \left|\frac{\d^2}{\d\tau^2}S\Gamma(x,\tau)\right|^2\ d\tau.
$$
Then by Corollary \ref{cor_gk}, we have
$$
\int_{\Gamma\cap B}\mathcal{S}^2\Gamma(x)\ d\mu(x) \le C_{\Gamma}\|\chi_B\|_{L^2}^2
$$
which shows that $\mathcal{S}^2\Gamma\in BMO(\Gamma)$ with the $BMO$ nom bounded by $C_\Gamma$. Let $C$ be a constant such that
$$
\sup_{\tau\in \R} \left\|\frac{\d^3}{\d\tau^3}S\Gamma(\cdot,\tau) \right\|_{L^\infty(\Gamma)}\le C.
$$
For each $x\in \Gamma_{\delta,N}$ and $\tau_i\in \T_{\delta}(x)$. Let $\epsilon = \delta/(2C)$ and $I_i = (\tau_i-\epsilon,\tau_i+\epsilon)$, then we have
$$
\left|\frac{\d^2}{\d\tau^2}S\Gamma(x,\tau) \right|\ge \frac{\delta}{2},\mbox{ for all } \tau\in I_i,
$$
and $I_i\cap \T_\Gamma(x) = \{\tau_i\}$ and the $\{I_i\}$ are disjoint. We have
\begin{equation*}
\begin{split}
\mathcal{S}^2\Gamma(x) &=  \ln(a)\int_{-\infty}^\infty \left|\frac{\d^2}{\d\tau^2}S\Gamma(x,\tau)\right|^2\ d\tau \ge \sum_{\tau_i\in \T_\delta(x)}\int_{I_i} \left|\frac{\d^2}{\d\tau^2}S\Gamma(x,\tau)\right|^2\ d\tau\\
&\ge (\delta/2)^2|I_i|(\#\T_\delta(x)) >  CN\delta^3,
\end{split}
\end{equation*}
for some new constant $C$.  This implies
\begin{equation*}
\begin{split}
\Gamma_{\delta,N}&\subset \{x\in \Gamma: \mathcal{S}^2\Gamma(x)>CN\delta^3\}\\
&\subset \left\{x\in \Gamma: \left|\mathcal{S}^2\Gamma(x)- \mathcal{S}_\Gamma\right|>CN\delta^3 - \mathcal{S}_\Gamma\right\},
\end{split}
\end{equation*}
where $\mathcal{S}_\Gamma = \frac{1}{\mu(\Gamma)}\int_{\Gamma}\mathcal{S}^2\Gamma(x)\ d\mu(x)$.

If $CN\delta^3>\mathcal{S}_\Gamma$, then by the John-Nirenberg Theorem \cite{john1961functions}, there exist positive constants $C_1'$ and $C_2'$ independent of $\Gamma$ such that
\begin{equation}\label{jn_eq1}
\mu\left(\left\{x\in \Gamma: \left|\mathcal{S}^2\Gamma(x)- \mathcal{S}_\Gamma\right|>CN\delta^3 - \mathcal{S}_\Gamma\right\}\right)\le |\Gamma|C_1'e^{-C_2'(CN\delta^3 -\mathcal{S}_\Gamma)/\|\mathcal{S}^2\Gamma\|_{BMO}}.
\end{equation}
On the other hand, if $CN\delta^3\le \mathcal{S}_\Gamma$, then \eqref{jn_eq1} still holds with $C_1'\ge 1$. Let 
$$
C_1 = |\Gamma|C_1'e^{\frac{C_2'\mathcal{S}_\Gamma}{\|\mathcal{S}^2\Gamma\|_{BMO}}},\mbox{ and } C_2 = \frac{C_2'C}{\|\mathcal{S}^2\Gamma\|_{BMO}},
$$
then 
$$
\mu(\Gamma_{\delta,N}) \le C_1e^{-C_2\delta^3N}.
$$
\end{proof}

Analogous to Corollary \ref{cor_beta_delta_ls}, for $x\in \Gamma$ and $\beta,\delta >0$, define
$$
\tau_{\beta,\delta}(x) = \left\{\tau\in \T_\Gamma(x): |S\Gamma(x,\tau)|>\beta, \left|\frac{\d^2}{\d \tau^2}S\Gamma(x,\tau)\right| >\delta\right\},
$$
and 
\begin{equation}\label{gamma_beta_delta_N}
\Gamma_{\beta,\delta,N} = \{x\in \Gamma: \#\tau_{\beta,\delta}(x) \ge N\}.
\end{equation}
A similar result as in Corollary \ref{cor_delta_ls} also holds for $|\Omega_{\beta,\delta,N}|$.
\begin{corollary}
Assume $\Gamma$ is as in Corollary \ref{cor_gamma_delta_n}. Then there exist constants $C_1$ and $C_2$ which depend on $\Gamma$ such that
$$
\mu(\Gamma_{\beta,\delta,N}) \le C_1e^{-C_2N\delta^2\alpha},
$$
where $\alpha = \min(\beta,\delta)$.
\end{corollary}

Analogous to \eqref{s*f}, we can also define the nontangential control of $S\Gamma(x,t)$ by
$$
S^*\Gamma(x,t) = \sup_{|x-y| < \frac{\sqrt{t}}{\pi}} |S\Gamma(y,t)|e^{-\pi|x-y|^2/t}. 
$$
The nontangential local scales can be defined as before using $S^*\Gamma(x,t)$.
\begin{definition}
The nontangential local scales of $\Gamma$ at $x$ is defined as the set
$$
T^*_\Gamma(x) = \{t\in (0,\infty): S^*\Gamma(x,y) \mbox{ is  a local maxima}\}.
$$
\end{definition}
It can also be shown in this case that $T^*_\Gamma(x)$ satisfy the dilating consistency property with $s=-2$, i.e. for all $\delta>0$,
\begin{equation}
T^*_{d_\delta\Gamma}(\delta x) =\left \{\delta^{-2}t: t\in T^*_\Gamma(x)\right\}.
\end{equation}

\section{Discussion}
In this section, we would to give more insights into our approach and discuss possible extensions.
\begin{enumerate}
\item The kernel (wavelet) $\psi_t$ dictates the type of local scales we see. In our case, $\psi_t = t\frac{\d K_t}{\d t}$ is symmetric and has zero mean and zero first moments. This implies that the local scales $t_i$'s at $x$ are points in time where $\Gamma$ have large deviations from a $d$-plane in the ball $B_{\sqrt{t}/\pi}(x)$. This is related to $\beta_p$ defined as \cite{david1993analysis}
\begin{equation}
\beta_p(x,t) = \inf_{P\in \mathcal{P}} \left\{\frac{1}{\mu(B(x,t))}\int_{\Gamma\cap B(x,t)} \left[t^{-1}dist(y,P) \right]^p  \right\}^{1/p},
\end{equation}
where the infimum is taken over the set $\mathcal{P}$ consisting of all $d$-plane $P$. Instead of considering $\mathcal{P}$, we can consider a different set of $d$-dimensional surfaces.

\item Let $\Gamma$ be a connected and bounded $d$-dimensional subset in $\R^n$, and suppose that it can be parametrized by $\Gamma = \{(f_1(r), \cdots, f_n(r)): r\in [0,1]^d\}$, where $f_i$ is continuous for each $i$. For each $i$, define
$$
u_i(r,t) = K_t*f_i(r) = \int_{\R^d} K_t(r-s)f_i(s)\ ds,
$$
where $K_t(r) = t^{-d/2}e^{-\pi |r|^2/t}$. For each $t>0$, denote by $\Gamma_t = (u_1(r,t), \cdots, u_n(r,t))$. We can think of $\Gamma_t$ as the diffused version of $\Gamma$ at scale $t$. Note that $u_i(t)$ depends on the parametrization of $f_i$. However, give a parametrization, this method provides a tool for obtaining a diffused $d$-plane $\Gamma_t$. Define $S\Gamma(r,t)$ by
$$
S\Gamma(r,t) = \left\|\left(t\frac{\d K_t}{\d t}*f_1(r), \cdots, t\frac{\d K_t}{\d t}*f_n(r)\right) \right\|.
$$
For each $r\in \R^d$, as before, the local scales of $\Gamma$ can be defined as the local maxima of $S\Gamma(r,t)$. This approach is considered by L-M  Reissell in \cite{reissell1996wavelet} and by P.L. Rosin \cite{rosin1998determining} to represent curves in a multiscale fashion using wavelets. Note that $u_i(r,t)$ can be viewed as a heat diffusion with the initial condition given by $f_i(r)$. One can also use nonlinear diffusions for $u_i(r,t)$.

\item In studying the local scales of $\Gamma$, we assume that the dimension $d$ of $\Gamma$ is known. One question would be: Is it possible to detect the dimension of $\Gamma$?. Better yet, is it possible to detect the local dimension in $\Gamma$ at different scales? Recently, this question is addressed by M. Maggioni and collaborators \cite{maggioni2010}.
\end{enumerate}

{\bf Acknowledgement:} The author is very grateful for the supports from NSF DMS 0809270 and ONR N000140910108.
\bibliography{biblioDef} \bibliographystyle{alpha}

\end{document}